\documentclass[submission,copyright,creativecommons]{eptcs}
\providecommand{\event}{SOS 2007} % Name of the event you are submitting to

\usepackage{iftex}
\usepackage[dvipsnames]{xcolor}

\ifpdf
  \usepackage{underscore}         % Only needed if you use pdflatex.
  \usepackage[T1]{fontenc}        % Recommended with pdflatex
\else
  \usepackage{breakurl}           % Not needed if you use pdflatex only.
\fi

\usepackage{enumerate,xspace}
\usepackage{amsmath,amssymb,wasysym}
\usepackage[all]{xy}
\usepackage{multicol}

\newtheorem{observation}{Remark}[section]
\newtheorem{lemma}[observation]{Lemma}  %%share counter with remark
\newtheorem{theorem}[observation]{Theorem}
\newtheorem{definition}[observation]{Definition}
\newtheorem{example}[observation]{Example}

\newtheorem{proposition}[observation]{Proposition} 
\newtheorem{corollary}[observation]{Corollary}

\title{Dagger-Drazin Inverses}
\author{Robin Cockett\thanks{Partially funded by NSERC.}
\institute{Department of Computer Science\\
University of Calgary, Canada}
\email{robin@ucalgary.ca}
\and
Jean-Simon Pacaud Lemay%\thanks{Financially supported by...}
\institute{School of Mathematical and Physical Sciences\\
Macquarie University, Australia}
\email{js.lemay@mq.edu.au}
\and
Priyaa Varshinee Srinivasan
\thanks{Funded  by the Estonian Research Council (MOB3JD1227).}
\institute{Department of Software Sciences\\
Tallinn University of Technology, Estonia}
\email{priyaavarshinee@gmail.com}
}

\def\event{QPL 2025}
\def\titlerunning{Dagger-Drazin Inverses}
\def\authorrunning{J.R.B. Cockett, J.-S. P. Lemay, \& P. V. Srinivasan}

\begin{document}
\maketitle

%15 pages (excluding references, acknowledgements, and appendices).

\begin{abstract}
Drazin inverses are a special kind of generalized inverses that can be defined for endomorphisms in any category. A natural question to ask is whether one can somehow extend the notion of Drazin inverse to arbitrary maps -- not simply endomorphisms. It turns out that this is possible and, indeed, natural to do so for dagger categories. This paper, thus, introduces the notion of a dagger-Drazin inverse, which is a new kind of generalized inverse appropriate for arbitrary maps in a dagger category. This inverse is closely related to the Drazin inverse, for having dagger-Drazin inverses is equivalent to asking that positive maps have Drazin inverses. Moreover, dagger-Drazin inverses are also closely related to Moore-Penrose inverses as we observe that a map has a Moore-Penrose inverse if and only if it is a Drazin inverse. Furthermore, we explain how Drazin inverses of opposing pairs correspond precisely to dagger-Drazin inverses in cofree dagger categories. We also give examples of dagger-Drazin inverses for matrices over (involutive) fields, bounded linear operators, and partial injections. 
\end{abstract}

\section{Introduction} 

Generalized inverses are of particular interest in matrix theory \cite{campbell2009generalized} since they allow one to solve linear equations $A \vec x = \vec y$ even when $A$ is not necessarily invertible. One such kind of generalized inverse that is of interest is the \emph{Drazin inverse} \cite[Chap 7]{campbell2009generalized} -- named after Michael P. Drazin who originally introduced the concept under the name ``pseudo-inverse'' \cite{drazin1958pseudo}. Every square matrix over a field has a Drazin inverse \cite[Def 7.2.3]{campbell2009generalized}, and can be used to solve linear differential equations \cite[Chap 9]{campbell2009generalized}. 

While Drazin inverses have been extensively studied in matrix theory, the notion of a Drazin inverse can more generally be defined in an arbitrary semigroup \cite{drazin1958pseudo,munn1961pseudo} and they have also been well studied in ring theory, as they are deeply connected to strong $\pi$-regularity \cite{azumaya1954strongly} and Fitting's Lemma/Theorem \cite{Fitting1933}. As such, one can consider Drazin inverses of matrices over a ring \cite{puystjens2004drazin}, linear endomorphisms of modules \cite{wang2017class}, and even bounded linear endomorphisms of Banach spaces \cite{King1977ANO} or Hilbert spaces \cite{wei2003representation}. The Drazin inverse is a useful tool for computation and thus has found many applications. In particular, for quantum information theory, Drazin inverses have been used to characterize quantum gates \cite{wu2014remarks}, for quantum Turing automata \cite{EPTCS143.2}, in quantum error mitigation \cite{cao2021nisq}, and quantum thermodynamics \cite{chen2021extrapolating,chen2023optimal,miller2019work}. 

There has also been renewed interest in studying Drazin inverses in categories. They arise naturally since for any object $A$ in a category $\mathbb{X}$, the homset $\mathbb{X}(A,A)$ is a semigroup (in fact a monoid) with respect to composition. As such, we may consider Drazin inverses in $\mathbb{X}(A,A)$. Thus, we may talk about Drazin inverses of endomorphisms in an arbitrary category (Def \ref{def:Drazin}). Drazin inverses in categories were introduced by Puystjens and Robinson in their 1987 paper \cite{robinson1987generalized}, where they were particularly interested in generalized inverses in \emph{additive} categories. Unfortunately, a deeper categorical approach to Drazin inverses was never initiated until very recently. In \cite{cockett2024drazin}, the authors further developed the theory of Drazin inverse in category theory. While, as one might expect, one can generalize ring theoretic Drazin inverse results to additive/abelian categories \cite[Sec 6]{cockett2024drazin}, the categorical perspective also broaden the theory of Drazin inverses in unexpected ways, such as relating Drazin inverses to the idempotent completion \cite[Sec 5]{cockett2024drazin} and to the notion of eventual image duality \cite{leinster2022eventual}. However, arriving with categorical eyes to the subject of Drazin inverses it is natural to want to have a notion of a Drazin inverse for arbitrary maps, not simply endomorphisms. The main purpose of this paper is to explain how this is quite natural in \emph{dagger} categories. 

Recall that a dagger category (which we write as $\dagger$-category) is a category equipped with an involution $\dagger$ on maps. The term ``dagger category'' was introduced by Selinger in \cite{selinger2007dagger} based on the use in physics of the symbol $\dagger$ for conjugate transpose. The theory of $\dagger$-categories is now a firmly established sub-field of category theory with a rich literature. In particular, $\dagger$-categories are a fundamental component of categorical quantum mechanics. In this paper, we introduce the notion of a \emph{dagger-Drazin inverse} (which we write as $\dagger$-Drazin inverses), which is a new kind of generalized inverse for maps in a $\dagger$-category (Def \ref{def:dagger-Drazin}). In particular, $\dagger$-Drazin inverses can be defined for maps of any type and not simply endomorphisms. While a $\dagger$-Drazin inverse of a map may not always exist, if it does it is unique (Prop \ref{prop:dag-Draz-unique}). As the name suggests, $\dagger$-Drazin inverses are in fact deeply connected to Drazin inverses (Sec \ref{sec:dag-Drazin-Drazin}). First of all, for self-adjoint maps (which are always endomorphisms), the notion of $\dagger$-Drazin inverses and Drazin inverses coincide (Lemma \ref{lem:self-adjoint}). Having $\dagger$-Drazin inverses it transpires  corresponds precisely to positive maps (which recall are always self-adjoint endomorphisms) having Drazin inverses (Thm \ref{thm:Drazin=dag-Drazin}). We give examples of $\dagger$-Drazin inverses for matrices over fields (Ex \ref{ex:matrices-field}), matrices over involutive fields (Ex \ref{ex:matrices-inv-field} \& \ref{ex:complex-matrices}), bounded linear operators between Hilbert spaces (Ex \ref{ex:Hilb} \& \ref{ex:FHILB}), and partial injections (Ex \ref{ex:inv-cat} \& \ref{ex:PINJ}). 

Moreover, $\dagger$-Drazin inverses are related to another kind of generalized inverse: the \emph{Moore-Penrose} inverse.  It is named after E. H. Moore and R. Penrose: Moore introduced the notion in 1920 \cite{moore1920reciprocal}, and it was independently rediscovered by Penrose in 1955 \cite{penrose1955generalized}. The Moore-Penrose inverse is arguably the most well-known generalized inverse, with many interesting applications \cite{baksalary2021moore}. In particular, every complex matrix has a Moore-Penrose inverse constructed using Singular Value Decomposition \cite[Chap 1]{campbell2009generalized}. Moore-Penrose inverses can also be defined in $\dagger$-categories, as was first done by Puystjens and Robinson in a series of papers in the 1980s \cite{puystjens1981moore, puystjens1984moore, robinson1985ep, robinson1987generalized, puystjens1990symmetric}, and later picked up again by the first and second named author in \cite{EPTCS384.10} for QPL2023. It turns out that one can give a novel alternative characterization of the Moore-Penrose inverse using $\dagger$-Drazin inverses. Indeed, a map has a Moore-Penrose inverse if and only if it is \emph{a} $\dagger$-Drazin inverse (Thm \ref{thm:MP=a-dag-Drazin}), and in this case the Moore-Penrose inverse is the $\dagger$-Drazin inverse. 

We also explain how $\dagger$-Drazin inverses are connected to the notion of Drazin inverses for opposing pairs \cite[Sec 7]{cockett2024drazin} (which was the authors' first attempt at generalizing Drazin inverses for arbitrary maps in a category) via \emph{cofree} $\dagger$-categories. For any category, there is a cofree $\dagger$-category over it \cite{heunen2009categorical}, and cofree $\dagger$-categories are used in reversible computation \cite{jacobs2009categorical}.  It turns out that Drazin inverses for opposing pairs correspond precisely to $\dagger$-Drazin inverses in cofree $\dagger$-categories (Thm \ref{thm:dag-Drazin-cofree}). 

Lastly, it is important to note the perhaps subtle distinction we make throughout this paper between \emph{being} a $\dagger$-Drazin inverse and \emph{having} a $\dagger$-Drazin inverse: the former is a stronger requirement than the latter. Thus, when a map \emph{has} a $\dagger$-Drazin inverse, it is not necessarily the case that it \emph{is} a $\dagger$-Drazin inverse.  In particular, a map with a $\dagger$-Drazin inverse is not in general the $\dagger$-Drazin inverse of its $\dagger$-Drazin inverse. On the other hand, every $\dagger$-Drazin inverse does have a $\dagger$-Drazin inverse and, in this case, it is also the $\dagger$-Drazin inverse of its $\dagger$-Drazin inverse (Prop \ref{prop:dag-Draz}.(\ref{prop:dag-Draz-dag-Draz})+(\ref{prop:dag-Draz-dag-Draz-dag-Drax})).

\section{Dagger-Drazin Inverses}

In this section we introduce the main novel concept of this paper: \emph{dagger-Drazin inverses}, which is a new kind of generalized inverses in a dagger category that is, as we will see in the next section, deeply connected to the notion of Drazin inverses. 

For an arbitrary category $\mathbb{X}$, we denote objects by capital letters $A,B,C$, etc. and maps by lowercase letters $f,g,x,y$, etc. Identity maps are denoted as $1_A: A \to A$.  Composition is written in \emph{diagrammatic order}, that is, the composition of a map $f: A \to B$ followed by $g: B \to C$ is denoted $fg: A \to C$. Then recall that a \textbf{dagger} on a category $\mathbb{X}$ is a contravariant functor $(\_)^\dagger: \mathbb{X} \to \mathbb{X}$ which is the identity on objects and involutive. We refer to $f^\dagger$ as the \textbf{adjoint} of $f$. A \textbf{$\dagger$-category} is a category $\mathbb{X}$ equipped with a  chosen dagger $\dagger$. Concretely, a $\dagger$-category can be described as a category $\mathbb{X}$ where for each map $f: A \to B$, there is a chosen map of dual type $f^\dagger: B \to A$ such that $1^\dagger_A = 1_A$, $(fg)^\dagger = g^\dagger f^\dagger$, and $(f^\dagger)^\dagger = f$. For a more in-depth introduction to dagger categories, we refer the reader to \cite{heunen2019categories}. 

\begin{definition}\label{def:dagger-Drazin} In a $\dagger$-category, a \textbf{$\dagger$-Drazin inverse} of a map $f: A \to B$ is a map of dual type $f^\partial: B \to A$ such that:  
%%\vspace{-10pt}
\begin{enumerate}[{\bf [$\text{D}^{\dagger}$.1]}]
\item There is a $k \in \mathbb{N}$ such that\footnote{By convention, for an endomorphism $x$, $x^0 = 1_A$.} $(f f^\dagger)^k f f^\partial = (f f^\dagger)^k$ and $f^\partial f (f^\dagger f)^k = (f^\dagger f)^k$; 
%\begin{multicols}{3}
\item $f^\partial f f^\partial = f^\partial$; 
%\columnbreak
\item $(ff^\partial)^\dagger = ff^\partial$;
%\columnbreak
\item $(f^\partial f)^\dagger = f^\partial f$.
%\end{multicols}
\end{enumerate}
%\vspace{-10pt}
%%\vspace{-10pt}
If $f$ has a $\dagger$-Drazin inverse, we say that $f$ is \textbf{$\dagger$-Drazin invertible} or simply \textbf{$\dagger$-Drazin}. We call the smallest natural number $k$ such that \textbf{[$\text{D}^{\dagger}$.1]} holds the \textbf{$\dagger$-Drazin index of $f$}, which we denote by $\mathsf{ind}^\partial(f)$. A \textbf{$\dagger$-Drazin category} is a $\dagger$-category such that every map is $\dagger$-Drazin. 
\end{definition}

\begin{proposition} \label{prop:dag-Draz-unique} In a $\dagger$-category,  if a map has a $\dagger$-Drazin inverse, then it is unique. 
\end{proposition}
\begin{proof} Suppose that a map $f$ has two possible $\dagger$-Drazin inverses $f^\partial$ and $f^\delta$. To show that these maps are indeed equal, it will be useful to first compute that for any $n \in \mathbb{N}$ the following equalities hold: 
\begin{align}\label{eq:useful-unique}     f^\partial = \left( f^\dagger f \right)^m f^\partial \left( ({f^\partial}^\dagger f^\partial)^\dagger \right)^m  && f^\delta = \left( f^\delta ({f^\delta}^\dagger )^\dagger \right)^m f^\delta (f f^\dagger)^m 
\end{align}
Clearly these hold for $m=0$. Then for $m \geq 1$, to compute the above equality on the left, we recursively use \textbf{[$\text{D}^{\dagger}$.2]}, \textbf{[$\text{D}^{\dagger}$.3]}, and \textbf{[$\text{D}^{\dagger}$.4]} as follows: 
\begin{gather*}
 f^\partial \stackrel{\text{\textbf{[$\text{D}^{\dagger}$.2]}}}{=} f^\partial f  f^\partial \stackrel{\text{\textbf{[$\text{D}^{\dagger}$.4]}}}{=} (f^\partial f)^\dagger f^\partial = f^\dagger {f^\partial}^\dagger f^\partial \stackrel{\text{\textbf{[$\text{D}^{\dagger}$.2]}}}{=} f^\dagger (f^\partial f  f^\partial)^\dagger f^\partial = f^\dagger {f^\partial}^\dagger f^\dagger {f^\partial}^\dagger f^\partial \\
 = f^\dagger (f f^\partial)^\dagger ({f^\partial}^\dagger f^\partial)^\dagger \stackrel{\text{\textbf{[$\text{D}^{\dagger}$.3]}}}{=} f^\dagger f f^\partial ({f^\partial}^\dagger f^\partial)^\dagger = ... = \left( f^\dagger f \right)^m f^\partial \left( ({f^\partial}^\dagger f^\partial)^\dagger \right)^m 
\end{gather*}
Similarly, we can also prove the right equality of (\ref{eq:useful-unique}). Now by \textbf{[$\text{D}^{\dagger}$.1]}, there exists a $k_1 \in \mathbb{N}$ such that $(f f^\dagger)^{k_1} f f^\partial = (f f^\dagger)^{k_1}$, and there also exists a $k_2 \in \mathbb{N}$ such that $f^\delta f (f^\dagger f)^{k_2} = (f^\dagger f)^{k_2}$. Then we compute:
\begin{gather*}
f^\delta \stackrel{\text{(\ref{eq:useful-unique})}}{=} \left( f^\delta ({f^\delta}^\dagger )^\dagger \right)^k f^\delta (f f^\dagger)^k  \stackrel{\text{\textbf{[$\text{D}^{\dagger}$.1]}}}{=}\left( f^\delta ({f^\delta}^\dagger )^\dagger \right)^k f^\delta (f f^\dagger)^k f f^\partial \stackrel{\text{(\ref{eq:useful-unique})}}{=} f^\delta f f^\partial  \\
\stackrel{\text{(\ref{eq:useful-unique})}}{=} f^\delta f \left( f^\dagger f \right)^{k_2} f^\partial \left( ({f^\partial}^\dagger f^\partial)^\dagger \right)^{k_2} \stackrel{\text{\textbf{[$\text{D}^{\dagger}$.1]}}}{=} \left( f^\dagger f \right)^{k_2} f^\partial \left( ({f^\partial}^\dagger f^\partial)^\dagger \right)^{k_2}  \stackrel{\text{(\ref{eq:useful-unique})}}{=} f^\partial
\end{gather*}
So $f^\delta= f^\partial$, and therefore we conclude that $\dagger$-Drazin inverses are unique. 
\end{proof}

Therefore, we can now speak of \emph{the} $\dagger$-Drazin inverse of a map. This also implies that for a $\dagger$-category, being $\dagger$-Drazin is a property rather than structure. We give examples of $\dagger$-Drazin inverses/categories in the following sections. We also remark that there is some redundancy in the above definition since \textbf{[$\text{D}^{\dagger}$.1]} can be expressed in other equivalent ways. 

\begin{lemma}\label{lem:D.5678} In a $\dagger$-category,  given maps $f: A \to B$ and $f^\partial: B \to A$ which together satisfy \textbf{[$\text{D}^{\dagger}$.3]} and \textbf{[$\text{D}^{\dagger}$.4]}, then \textbf{[$\text{D}^{\dagger}$.1]} is equivalent to any of the following statements: 
\begin{enumerate}[{\bf [$\text{D}^{\dagger}$.1]}]
\setcounter{enumi}{4}
\item There is a $j \in \mathbb{N}$ such that $f f^\partial (f f^\dagger)^j  = (f f^\dagger)^j$;  
\item There is a $k \in \mathbb{N}$ such that $(f f^\dagger)^k f f^\partial = (f f^\dagger)^k$;
\item There is an $m \in \mathbb{N}$ such that $f^\partial f (f^\dagger f)^m = (f^\dagger f)^m$;
\item  There is an $n \in \mathbb{N}$ such that $(f^\dagger f)^n f^\partial f = (f^\dagger f)^n$.  
\end{enumerate}
Moreover if $f$ is $\dagger$-Drazin, then {\bf [$\text{D}^{\dagger}$.5]}--{\bf [$\text{D}^{\dagger}$.8]} all hold for all $k \geq \mathsf{ind}^\partial(f)$. 
\end{lemma}
\begin{proof} We will first explain why {\bf [$\text{D}^{\dagger}$.5]}--{\bf [$\text{D}^{\dagger}$.8]} are all equivalent. For {\bf [$\text{D}^{\dagger}$.5]} $\Rightarrow$ {\bf [$\text{D}^{\dagger}$.6]}, suppose we have that $f f^\partial (f f^\dagger)^j  = (f f^\dagger)^j$ for some $j$. Then setting $k=j$, we compute that: 
\begin{gather*}
 (f f^\dagger)^j f f^\partial \stackrel{\text{\textbf{[$\text{D}^{\dagger}$.3]}}}{=} (f f^\dagger)^j (f f^\partial)^\dagger = \left( (f f^\dagger)^\dagger \right)^j  (f f^\partial)^\dagger =  \left( f f^\partial (f f^\dagger)^j  \right)^\dagger \stackrel{\text{\textbf{[$\text{D}^{\dagger}$.5]}}}{=} \left( (f f^\dagger)^\dagger \right)^j = (f f^\dagger)^j
\end{gather*}
For {\bf [$\text{D}^{\dagger}$.6]} $\Rightarrow$ {\bf [$\text{D}^{\dagger}$.7]}, suppose we have that $(f f^\dagger)^k f f^\partial = (f f^\dagger)^k$ for some $k$. Setting $m=k+1$, we compute that: 
\begin{gather*}
f^\partial f (f^\dagger f)^{k+1} \stackrel{\text{\textbf{[$\text{D}^{\dagger}$.4]}}}{=} (f^\partial f)^\dagger (f^\dagger f)^{k+1} = (f^\partial f)^\dagger \left( (f^\dagger f)^\dagger \right)^{k+1} = \left( (f^\dagger f)^{k+1} f^\partial f \right)^\dagger = \left( f^\dagger (f f^\dagger)^{k} f f^\partial f \right)^\dagger \\
\stackrel{\text{\textbf{[$\text{D}^{\dagger}$.6]}}}{=} \left( f^\dagger (f f^\dagger)^{k} f \right)^\dagger = \left( (f^\dagger f)^{k+1}  \right)^\dagger = \left( (f^\dagger f)^\dagger \right)^{k+1} = (f^\dagger f)^{k+1} 
\end{gather*}
For the remaining parts, {\bf [$\text{D}^{\dagger}$.7]} $\Rightarrow$ {\bf [$\text{D}^{\dagger}$.8]} is proven is similar fashion to how {\bf [$\text{D}^{\dagger}$.5]} $\Rightarrow$ {\bf [$\text{D}^{\dagger}$.6]} was proved (though using \textbf{[$\text{D}^{\dagger}$.4]} instead), while {\bf [$\text{D}^{\dagger}$.8]} $\Rightarrow$ {\bf [$\text{D}^{\dagger}$.5]} is proven similar fashion to how we showed {\bf [$\text{D}^{\dagger}$.6]} $\Rightarrow$ {\bf [$\text{D}^{\dagger}$.7]} (but where we use \textbf{[$\text{D}^{\dagger}$.3]} instead). Next note that {\bf [$\text{D}^{\dagger}$.1]} is precisely \textbf{[$\text{D}^{\dagger}$.6]} and \textbf{[$\text{D}^{\dagger}$.7]} together, and therefore is indeed equivalent to any of {\bf [$\text{D}^{\dagger}$.5]}--{\bf [$\text{D}^{\dagger}$.8]}. Lastly, it is straightforward to check that if $f$ is $\dagger$-Drazin, then {\bf [$\text{D}^{\dagger}$.5]}--{\bf [$\text{D}^{\dagger}$.8]} all hold for all $k \geq \mathsf{ind}^\partial(f)$.
\end{proof}

\begin{corollary}  In a $\dagger$-category,  for a map $f: A \to B$, a map of dual type $f^\partial: B \to A$ is the $\dagger$-Drazin inverse of $f$ if and only if $f$ and $f^\partial$ satisfies \textbf{[$\text{D}^{\dagger}$.2]}, \textbf{[$\text{D}^{\dagger}$.3]}, \textbf{[$\text{D}^{\dagger}$.4]}, and any one of {\bf [$\text{D}^{\dagger}$.5]}--{\bf [$\text{D}^{\dagger}$.8]}. 
\end{corollary}

In any $\dagger$-category, there are certain maps which are $\dagger$-Drazin by default: 

\begin{lemma}\label{lem:dag-Draz-maps} In a $\dagger$-category, 
    \begin{enumerate}[(i)]
    \item Identity maps $1_A$ are $\dagger$-Drazin where $1_A^\partial = 1_A$ and $\mathsf{ind}^\partial(1_A)=0$; 
\item \label{lem:dag-Draz-iso} If $f$ is an isomorphism then $f$ is $\dagger$-Drazin where $f^\partial = f^{-1}$ and $\mathsf{ind}^\partial(f)=0$;
\item \label{lem:dag-Draz-partial-isom} If $f$ is a partial isometry (i.e. $ff^\dagger f = f$) then $f$ is $\dagger$-Drazin where $f^\partial = f^\dagger$ and $\mathsf{ind}^\partial(f)\leq 1$; 
\item \label{lem:dag-Draz-unitary} If $f$ is unitary (i.e. $ff^\dagger =1$ and $f^\dagger f=1$) then $f$ is $\dagger$-Drazin where $f^\partial = f^\dagger$ and $\mathsf{ind}^\partial(f)=0$;
\item \label{lem:dag-Draz-dag-idempotent} If $e$ is a $\dagger$-idempotent (i.e. $ee=e$ and $e^\dagger =e$) then $e$ is $\dagger$-Drazin where $e^\partial =e$.  Furthermore, when $e \neq 1_A$, $\mathsf{ind}^\partial(e)=1$ and $\mathsf{ind}^\partial(e)=0$ precisely when $e=1_A$. 
\end{enumerate}
\end{lemma}
\begin{proof} These are straightforward to check, so we leave them as an exercise for the reader. \end{proof}

Having a dagger-inverse has a number of consequences: 

\begin{proposition}\label{prop:dag-Draz} In a $\dagger$-category,  if a map $f$ is $\dagger$-Drazin then: 
\begin{enumerate}[(i)]
\item \label{prop:D3D4-expand} $f f^\partial =  {f^\partial}^\dagger f^\dagger$ and $f^\partial f = f^\dagger {f^\partial}^\dagger $;
\item \label{prop:D2D3D4} $f^\partial = f^\dagger {f^\partial}^\dagger f^\partial = f^\partial {f^\partial}^\dagger f^\dagger$ and  ${f^\partial}^\dagger = f f^\partial {f^\partial}^\dagger  = {f^\partial}^\dagger f^\partial  f$; 
\item \label{prop:dag-Draz-adj} $f^\dagger$ is also $\dagger$-Drazin where ${f^\dagger}^\partial = {f^\partial}^\dagger$ and $\mathsf{ind}^\partial(f) = \mathsf{ind}^\partial(f^\dagger)$;
\item \label{prop:dag-Draz-dag-Draz} $f^\partial$ is $\dagger$-Drazin where $f^{\partial \partial} = f f^\partial f$ and $\mathsf{ind}^\partial(f^\partial) \leq 1$;
\item \label{prop:dag-Draz-dag-Draz-dag-Drax} $f^{\partial \partial \partial} = f^\partial$;
\item \label{prop:dag-Draz-dag-Draz-partial-isometry} If $f^\partial = f^\dagger$ then $f$ is a partial isometry; 
\item  \label{prop:dag-Draz-dag-Draz-iso} If $\mathsf{ind}^\partial(f)=0$ then $f$ is an isomorphism and $f^{-1} = f^\partial$;
\item \label{prop:dag-Draz-dag-Draz-unitary} If $f^\partial = f^\dagger$ and $\mathsf{ind}^\partial(f)=0$ then $f$ is unitary; 
\item \label{prop:dag-Draz-ffdag} $ff^\dagger$ and $f^\dagger f$ are $\dagger$-Drazin where $(f f^\dagger)^\partial = {f^\dagger}^\partial f^\partial$ and $(f^\dagger f)^\partial = f^\partial {f^\dagger}^\partial$;
\item \label{prop:dag-Draz-ffpartial} $ff^\partial$ and $f^\partial f$ are partial isometries and so $\dagger$-Drazin where $(ff^\partial)^\partial = ff^\partial$ and $(f^\partial f)^\partial = f^\partial f$. 
\end{enumerate}
\end{proposition}
\begin{proof} For (\ref{prop:D3D4-expand}), this is just expanding out \textbf{[$\text{D}^{\dagger}$.3]} and \textbf{[$\text{D}^{\dagger}$.4]} expanded out. For (\ref{prop:D2D3D4}), this follows from \textbf{[$\text{D}^{\dagger}$.2]} and (\ref{prop:D3D4-expand}). For (\ref{prop:dag-Draz-adj}), suppose that $\mathsf{ind}^\partial(f)=k$. We show that ${f^\dagger}^\partial := {f^\partial}^\dagger$ satisfies the necessary axioms to be a $\dagger$-Drazin inverse of $f^\dagger$: 
\begin{enumerate}[{\bf [$\text{D}^{\dagger}$.1]}]
\setcounter{enumi}{5}
\item $(f^\dagger f)^k f^\dagger {f^\dagger}^\partial \stackrel{\text{def.}}{=}  (f^\dagger f)^k f^\dagger {f^\partial}^\dagger \stackrel{\text{(\ref{prop:D3D4-expand})}}{=}(f^\dagger f)^k f^\partial f \stackrel{\text{\textbf{[$\text{D}^{\dagger}$.8]}}}{=} (f^\dagger f)^k$
\end{enumerate}
    \begin{enumerate}[{\bf [$\text{D}^{\dagger}$.1]}]
\setcounter{enumi}{1}
\item ${f^\dagger}^\partial f^\dagger {f^\dagger}^\partial \stackrel{\text{def.}}{=} {f^\partial}^\dagger f^\dagger {f^\partial}^\dagger = \left( f^\partial f f^\partial \right)^\dagger \stackrel{\text{\textbf{[$\text{D}^{\dagger}$.2]}}}{=} {f^\partial}^\dagger \stackrel{\text{def.}}{=} {f^\dagger}^\partial$
\item $\left(f^\dagger {f^\dagger}^\partial\right)^\dagger \stackrel{\text{def.}}{=} \left( f^\dagger {f^\partial}^\dagger\right)^\dagger \stackrel{\text{(\ref{prop:D3D4-expand})}}{=} (f^\partial f)^{\dagger}  \stackrel{\text{\textbf{[$\text{D}^{\dagger}$.4]}}}{=} f^\partial f \stackrel{\text{(\ref{prop:D3D4-expand})}}{=} f^\dagger {f^\partial}^\dagger \stackrel{\text{def.}}{=} f^\dagger {f^\dagger}^\partial$
\item $\left( {f^\dagger}^\partial f^\dagger\right)^\dagger \stackrel{\text{def.}}{=} \left(  {f^\partial}^\dagger f^\dagger\right)^\dagger \stackrel{\text{(\ref{prop:D3D4-expand})}}{=} (f f^\partial)^\dagger  \stackrel{\text{\textbf{[$\text{D}^{\dagger}$.3]}}}{=} f f^\partial  \stackrel{\text{(\ref{prop:D3D4-expand})}}{=}  {f^\partial}^\dagger f^\dagger \stackrel{\text{def.}}{=}  {f^\dagger}^\partial f^\dagger$
\end{enumerate}
So we conclude that $f^\dagger$ is $\dagger$-Drazin. It is straightforward to check that $\mathsf{ind}^\partial(f) = \mathsf{ind}^\partial(f^\dagger)$. For (\ref{prop:dag-Draz-dag-Draz}), we show that $f^{\partial \partial} := f f^\partial f$ satisfies the necessary axioms to be $\dagger$-Drazin inverse of $f^\partial$. 
\begin{enumerate}[{\bf [$\text{D}^{\dagger}$.1]}]
\setcounter{enumi}{5}
\item $f^\partial {f^\partial}^\dagger f^\partial f^{\partial \partial} \stackrel{\text{def.}}{=}  f^\partial  {f^\partial}^\dagger f^\partial f f^\partial f \stackrel{\text{(\ref{prop:D3D4-expand})}}{=} 
    f^\partial {f^\partial}^\dagger f^\dagger {f^\partial}^\dagger =  f^\partial (f^\partial f f^\partial)^\dagger \stackrel{\text{\textbf{[$\text{D}^{\dagger}$.2]}}}{=}  f^\partial {f^\partial}^\dagger $
\end{enumerate}
\begin{enumerate}[{\bf [$\text{D}^{\dagger}$.1]}]
\setcounter{enumi}{1}
\item $f^{\partial \partial}  f^\partial  f^{\partial \partial} \stackrel{\text{def.}}{=} f f^\partial f f^\partial f f^\partial f \stackrel{\text{\textbf{[$\text{D}^{\dagger}$.2]}}}{=} f f^\partial f f^\partial f \stackrel{\text{\textbf{[$\text{D}^{\dagger}$.2]}}}{=} f f^\partial f \stackrel{\text{def.}}{=}  f^{\partial \partial}$
\item $(f^\partial f^{\partial \partial})^\dagger \stackrel{\text{def.}}{=} (f^\partial f f^\partial f)^\dagger \stackrel{\text{\textbf{[$\text{D}^{\dagger}$.2]}}}{=} (f^\partial f)^\dagger \stackrel{\text{\textbf{[$\text{D}^{\dagger}$.4]}}}{=} f^\partial f \stackrel{\text{\textbf{[$\text{D}^{\dagger}$.2]}}}{=} f^\partial f f^\partial f \stackrel{\text{def.}}{=} f^\partial f^{\partial \partial}$
\item $(f^{\partial \partial} f^\partial)^\dagger \stackrel{\text{def.}}{=} ( f f^\partial f f^\partial)^\dagger \stackrel{\text{\textbf{[$\text{D}^{\dagger}$.2]}}}{=} (f f^\partial )^\dagger \stackrel{\text{\textbf{[$\text{D}^{\dagger}$.3]}}}{=} f f^\partial  \stackrel{\text{\textbf{[$\text{D}^{\dagger}$.2]}}}{=} f f^\partial f f^\partial  \stackrel{\text{def.}}{=}  f^{\partial \partial} f^\partial$
\end{enumerate}
So we conclude that $f^{\partial \partial} = f f^\partial f$ is the $\dagger$-Drazin inverse of $f^\partial$. It easy to check that ${\mathsf{ind}^\partial(f^\partial) \leq 1}$. Now for (\ref{prop:dag-Draz-dag-Draz-dag-Drax}), applying (\ref{prop:dag-Draz-dag-Draz}) to $f^\partial$ we get that $f^{\partial \partial}$ is also $\dagger$-Drazin where $f^{\partial \partial\partial} = f^\partial f^{\partial \partial} f^\partial$. Expanding this out further, we compute that $f^{\partial \partial\partial} \stackrel{\text{def.}}{=} f^\partial f^{\partial \partial} f^\partial \stackrel{\text{def.}}{=}  f^\partial f f^\partial f f^\partial \stackrel{\text{\textbf{[$\text{D}^{\dagger}$.2]}}}{=}  f^\partial f f^\partial \stackrel{\text{\textbf{[$\text{D}^{\dagger}$.2]}}}{=} f^\partial$. So $f^\partial$ is the $\dagger$-Drazin inverse of $f^{\partial \partial}$. For (\ref{prop:dag-Draz-dag-Draz-partial-isometry}), suppose that $f^\partial = f^\dagger$. Then \textbf{[$\text{D}^{\dagger}$.2]} tells us that $f^\dagger f f^\dagger = f^\dagger$, which says that $f^\dagger$ is a partial isometry. Then applying the dagger to this equality gives $f f^\dagger f = f$, and so we have that $f$ is a partial isometry. For (\ref {prop:dag-Draz-dag-Draz-iso}), suppose that $\mathsf{ind}^\partial(f)=0$. This means that $ff^\partial = 1$ and $f^\partial f=1$, which is precisely that $f$ is an isomorphism with inverse $f^{-1} = f^\partial$. Then (\ref{prop:dag-Draz-dag-Draz-unitary}) immediately follows from (\ref{prop:dag-Draz-dag-Draz-partial-isometry}) and (\ref {prop:dag-Draz-dag-Draz-iso}). We will prove (\ref{prop:dag-Draz-ffdag}) in Thm \ref{thm:Drazin=dag-Drazin}. For (\ref{prop:dag-Draz-ffpartial}), note that we can easily check that $f f^\partial (f f^\partial)^\dagger f f^\partial \stackrel{\text{\textbf{[$\text{D}^{\dagger}$.3]}}}{=}  f f^\partial f f^\partial f f^\partial \stackrel{\text{\textbf{[$\text{D}^{\dagger}$.2]}}}{=}  f f^\partial f f^\partial \stackrel{\text{\textbf{[$\text{D}^{\dagger}$.2]}}}{=} f f^\partial$, and similarly that $f^\partial f ( f^\partial f)^\dagger  f^\partial f = f^\partial f$. So both $f f^\partial$ and $f^\partial f$ are partial isometries, so by Lemma \ref{lem:dag-Draz-maps}.(\ref{lem:dag-Draz-partial-isom}), they are both $\dagger$-Drazin and their adjoints are their $\dagger$-Drazin inverse. However by \textbf{[$\text{D}^{\dagger}$.3]} and \textbf{[$\text{D}^{\dagger}$.4]}, we get $(ff^\partial)^\partial = ff^\partial$ and $(f^\partial f)^\partial = f^\partial f$. 
\end{proof}

Unfortunately, like other generalized inverses (such as the Drazin inverse and the Moore-Penrose inverse), $\dagger$-Drazin inverses do not necessarily play well with composition. Indeed, even if $f$ and $g$ are $\dagger$-Drazin, $fg$ is may not be $\dagger$-Drazin, and even if it was, $(fg)^\partial$ may not equal $g^\partial f^\partial$. 

We conclude this section by discussing the special case of having a $\dagger$-Drazin index less than or equal to $1$. In classical Drazin inverse theory, having Drazin index less than or equal to $1$ corresponds precisely to another special kind of generalized inverse called a \emph{group inverse} \cite[Def 7.2.4]{campbell2009generalized} -- which can be defined without explicitly mentioning Drazin inverses. Borrowing this terminology, we introduce the notion of a \emph{$\dagger$-group inverse}. 

\begin{definition}
    In a $\dagger$-category, a {\bf $\dagger$-group inverse} of a map $f: A \to B$ is a map of dual type $f^\partial: B \to A$ such that:
    %\vspace{-10pt}
\begin{enumerate}[{\bf [$\text{G}^{\dagger}$.{1}.a]}]
%\begin{multicols}{2}
\item $ff^\dagger ff^\partial = ff^\dagger$ 
%\columnbreak
\item $f^\partial f f^\dagger f = f^\dagger f$
%\end{multicols}
\end{enumerate}
\begin{enumerate}[{\bf [$\text{G}^{\dagger}$.1]}]
\setcounter{enumi}{1}
%\begin{multicols}{3}
\item $f^\partial f f^\partial = f^\partial$; 
%\columnbreak
\item $(ff^\partial)^\dagger = ff^\partial$;
%\columnbreak
\item $(f^\partial f)^\dagger = f^\partial f$.
%\end{multicols}
\end{enumerate}
%\vspace{-10pt}
\end{definition}

\begin{lemma}\label{Lem:dagger-group inverse}
    In a $\dagger$-category, a map $f: A \to B$ has a $\dagger$-group inverse if and only if $f$ is $\dagger$-Drazin with $\mathsf{ind}^\partial(f) \leq 1$. Moreover, in this case, the $\dagger$-group inverse coincides with the $\dagger$-Drazin inverse, and the following equalities also hold:
    %\vspace{-8pt}
    \begin{enumerate}[{\bf [$\text{G}^{\dagger}$.{1}.a]}]
    \setcounter{enumi}{2}
%\begin{multicols}{2}
\item $ff^\partial ff^\dagger  = ff^\dagger$ 
%\columnbreak
\item $f^\dagger f f^\partial f  = f^\dagger f$
%\end{multicols}
\end{enumerate}
%\vspace{-10pt}
\end{lemma}
\begin{proof} For the $\Rightarrow$ direction, suppose that $f$ has a $\dagger$-group inverse $f^\partial$. Then we clearly see that {\bf [$\text{G}^{\dagger}$.{1}.a]} and {\bf [$\text{G}^{\dagger}$.{1}.b]} combined is {\bf [$\text{D}^{\dagger}$.1]} for $k=1$, while {\bf [$\text{G}^{\dagger}$.2]}, {\bf [$\text{G}^{\dagger}$.3]}, and {\bf [$\text{G}^{\dagger}$.4]} are precisely {\bf [$\text{D}^{\dagger}$.2]}, {\bf [$\text{D}^{\dagger}$.3]}, and {\bf [$\text{D}^{\dagger}$.4]} respectively. So $f$ is $\dagger$-Drazin with $\dagger$-Drazin being its $\dagger$-group inverse, and moreover {\bf [$\text{G}^{\dagger}$.{1}.a]} and {\bf [$\text{G}^{\dagger}$.{1}.b]} tells us precisely that $\mathsf{ind}^\partial(f) \leq 1$. For the $\Leftarrow$ direction, suppose that $f$ is $\dagger$-Drazin with $\dagger$-Drazin inverse $f^\partial$ $\mathsf{ind}^\partial(f) \leq 1$. Then clearly, {\bf [$\text{D}^{\dagger}$.2]}, {\bf [$\text{D}^{\dagger}$.3]}, and {\bf [$\text{D}^{\dagger}$.4]} are the same as {\bf [$\text{G}^{\dagger}$.2]}, {\bf [$\text{G}^{\dagger}$.3]}, and {\bf [$\text{G}^{\dagger}$.4]} respectively, while $\mathsf{ind}^\partial(f) \leq 1$ tells us that {\bf [$\text{D}^{\dagger}$.1]} holds for $k=1$ which is precisely {\bf [$\text{G}^{\dagger}$.{1}.a]} and {\bf [$\text{G}^{\dagger}$.{1}.b]}. Lastly, it follows from Lemma \ref{lem:D.5678} that {\bf [$\text{G}^{\dagger}$.{1}.c]} and {\bf [$\text{G}^{\dagger}$.{1}.d]} hold, as they are {\bf [$\text{D}^{\dagger}$.5]} and {\bf [$\text{D}^{\dagger}$.8]} for $k=1$. 
\end{proof}

Since $\dagger$-group inverses are special cases of $\dagger$-Drazin inverses, they are unique as well. Moreover, Lemma \ref{lem:dag-Draz-maps}.(\ref{lem:dag-Draz-iso}), (\ref{lem:dag-Draz-partial-isom}), and (\ref{lem:dag-Draz-dag-idempotent}) tells us that isomorphisms (in particular, identity maps and unitary maps), partial isometries, and $\dagger$-idempotents all have $\dagger$-group inverses. Moreover, observe that by combining Lemma \ref{lem:dag-Draz-maps}.(\ref{lem:dag-Draz-iso}) and Prop \ref{prop:dag-Draz}.(\ref{prop:dag-Draz-dag-Draz-iso}), we see that a map $f$ is $\dagger$-Drazin with $\mathsf{ind}^\partial(f)=0$ if and only if $f$ is an isomorphism, and in this case its $\dagger$-group inverse is its inverse, $f^\partial = f^{-1}$. As such, the maps with $\dagger$-Drazin index $1$ are precisely those that have a $\dagger$-group inverse and are \emph{not} an isomorphism. 

\section{Comparing Dagger-Drazin and Drazin}\label{sec:dag-Drazin-Drazin}

In this section, we explain the relationship between Drazin inverses and $\dagger$-Drazin inverses. In particular, the main result of this section is that having $\dagger$-Drazin inverses is equivalent to having Drazin inverses of positive maps. For a full detailed introduction to Drazin inverses in categories, see \cite{cockett2024drazin}. 

\begin{definition}  \label{def:Drazin} \cite[Def 2.1]{cockett2024drazin} In a category, a \textbf{Drazin inverse} of an endomorphism $x: A \to A$ is an endomorphism $x^D: A \to A$ such that:  
%\vspace{-10pt}
\begin{enumerate}[{\bf [$\text{D}$.1]}]
%\begin{multicols}{2}
\item There is a $k \in \mathbb{N}$ such that $x^{k+1} x^D = x^k$;
%\columnbreak
%\begin{multicols}{2} 
\item $x^D x x^D = x^D$; 
%\columnbreak
\item $x x^D = x^D x$. 
%\end{multicols}
%\end{multicols}
\end{enumerate}
%\vspace{-10pt}
If $x$ has a Drazin inverse, we say that $x$ is \textbf{Drazin invertible} or simply \textbf{Drazin}. We call the smallest natural number $k$ such that \textbf{[D.1]} holds the \textbf{Drazin index of $x$}, which we denote by $\mathsf{ind}^D(x)$. A \textbf{Drazin category} is a category such that every map is Drazin. 
\end{definition}

Drazin inverses are \emph{unique} \cite[Prop 2.3]{cockett2024drazin}, so again we may speak of \emph{the} Drazin inverse of a map and that being Drazin is a property rather than structure. Many examples and properties of Drazin inverses can be found in \cite{cockett2024drazin}, including the fact that having a Drazin index less than or equal to $1$ corresponds precisely to the notion of a \emph{group inverse}: 

\begin{definition}\label{def:groupinv} \cite[Def 3.8]{cockett2024drazin} In a category, a \textbf{group inverse} of an endomorphism $x: A \to A$ is an endomorphism $x^D: A \to A$ such that:
\begin{enumerate}[{\bf [$\text{G}$.1]}]
\item $x x^D x = x$;
\item $x^D x x ^D = x^D$;
\item $x^D x = x x^D$
\end{enumerate}
\end{definition}

In \cite[Lemma 3.9]{cockett2024drazin} it is proven that $x$ is Drazin with $\mathsf{ind}(x) \leq 1$ if and only if $x$ has a group inverse . As such, a group inverse is the Drazin inverse for endomorphisms with Drazin index less than or equal to $1$. 

We now add dagger structure back into the story. Thankfully, Drazin inverses behave well in a $\dagger$-category. 

\begin{lemma}\label{lem:Draz-dag} \cite[Lemma 7.22]{cockett2024drazin} In a $\dagger$-category, an endomorphism $x$ is Drazin if and only if $x^\dagger$ is Drazin. Moreover, if $x$ is Drazin then ${x^\dagger}^D = {x^D}^\dagger$. 
\end{lemma}

Let us now begin exploring the connection between Drazin inverses and $\dagger$-Drazin inverses. We first observe that for \emph{self-adjoint} maps\footnote{Note that self-adjoint maps are always endomorphisms.}, the two notions coincide. 

\begin{lemma}\label{lem:self-adjoint} In a $\dagger$-category, a self-adjoint map $x$ (that is, $x=x^\dagger$) is $\dagger$-Drazin if and only if $x$ is Drazin. Moreover, if $x$ is self-adjoint and ($\dagger$-)Drazin, then its ($\dagger$-)Drazin inverse is also self-adjoint and $x^D = x^\partial$. 
\end{lemma}
\begin{proof} Let $x$ be self-adjoint, so $x = x^\dagger$. For $\Rightarrow$, suppose that $x$ is $\dagger$-Drazin. We will show that its $\dagger$-Drazin inverse $x^\partial$ is also its Drazin inverse. Now since $x$ is self-adjoint, we can rewrite \textbf{[$\text{D}^{\dagger}$.6]} as saying there exists a $k \in \mathbb{N}$ such that $x^{2k+1} x^\partial = x^{2k}$, which implies that \textbf{[D.1]} holds. Note that \textbf{[$\text{D}^{\dagger}$.2]} and \textbf{[D.2]} are the same. Next, observe that by Prop \ref{prop:dag-Draz}.(\ref{prop:dag-Draz-adj}), we have that ${x^\partial}^\dagger = {x^\dagger}^\partial = x^\partial$, so $x^\partial$ is self-adjoint. Therefore, we may rewrite both \textbf{[$\text{D}^{\dagger}$.3]} and \textbf{[$\text{D}^{\dagger}$.4]} as $x x^\partial = x^\partial x$, which is precisely \textbf{[D.3]}. Therefore, we conclude that $x$ is Drazin with Drazin inverse $x^D= x^\partial$. For $\Leftarrow$, suppose that $x$ is Drazin. We will show that its Drazin inverse $x^D$ is also its $\dagger$-Drazin inverse. Now if $\mathsf{ind}^D(x)=k$, \cite[Lemma 2.2.(i)]{cockett2024drazin} tells us that for all $n \geq k$, we have that $x^{n+1} x^D = x^n$. In particular this gives us that $x^{2k+1} x^D = x^{2k}$, which since $x$ is self-adjoint is precisely \textbf{[$\text{D}^{\dagger}$.6]}. As before, note that \textbf{[D.2]} and \textbf{[$\text{D}^{\dagger}$.2]} are the same. Next, by Lemma \ref{lem:Draz-dag}, since $x$ is self-adjoint, $x^D$ is also self-adjoint. Then by self-adjointness and \textbf{[D.3]}, we have that $(xx^D)^\dagger = x^D x = x x^D = (x^D x)^\dagger$, which gives both 
    \textbf{[$\text{D}^{\dagger}$.3]} and \textbf{[$\text{D}^{\dagger}$.4]}. Therefore, we conclude that $x$ is $\dagger$-Drazin with $\dagger$-Drazin inverse $x^\partial = x^D$. 
\end{proof}

An important class of self-adjoint maps are the \emph{positive} maps. Recall that in a $\dagger$-category, a \textbf{positive map} is an endomorphism $p: A \to A$ such that there exists a map $f: A \to B$ such that $p = f f^\dagger$. 

\begin{corollary}\label{cor:positive} In a $\dagger$-category, a positive map is $\dagger$-Drazin if and only if it is Drazin. 
\end{corollary}

Now for any map $f$ in a $\dagger$-category, we have two positive maps associated to it: $ff^\dagger$ and $f^\dagger f$. We will now show that $f$ being $\dagger$-Drazin is equivalent to either one of these positive maps being Drazin. 

\begin{theorem}\label{thm:Drazin=dag-Drazin}  In a $\dagger$-category,  for a map $f$, the following are equivalent: 
%\vspace{-5pt}
\begin{enumerate}[(i)]
%\begin{multicols}{4}
\item $f$ is $\dagger$-Drazin;
%\columnbreak
\item $f^\dagger$ is $\dagger$-Drazin;
%\columnbreak
\item $f f^\dagger$ is ($\dagger$-)Drazin; 
%\columnbreak
\item $f^\dagger f$ is ($\dagger$-)Drazin.
%\end{multicols}
\end{enumerate}
%\vspace{-10pt}
Moreover, if $f$ is $\dagger$-Drazin, then:
\begin{enumerate}[(i)]
\setcounter{enumi}{5}
\item \label{thm:vi} $f^\partial = f^\dagger (f f^\dagger)^D = (f^\dagger f)^D f^\dagger$ and ${f^\dagger}^\partial = f (f^\dagger f)^D  =  (f f^\dagger)^D f$;
\item \label{thm:vii} $(f f^\dagger)^D = {f^\dagger}^\partial f^\partial$ and $(f^\dagger f)^D = f^\partial {f^\dagger}^\partial$;
\item \label{thm:viii} $\mathsf{ind}^\partial(f) = \mathsf{max}\lbrace \mathsf{ind}^D(ff^\dagger), \mathsf{ind}^D(f^\dagger f) \rbrace$. 
\end{enumerate}
\end{theorem}
\begin{proof} First note that $(i) \Leftrightarrow (ii)$ follows from Prop \ref{prop:dag-Draz}.(\ref{prop:dag-Draz-adj}). On the other hand, $(iii) \Leftrightarrow (iv)$ follows from the fact in any category, for maps $g: A \to B$ and $h: B \to A$, $gh$ is Drazin if and only if $hg$ is Drazin \cite[Lemma 7.3]{cockett2024drazin}. Then combining this fact with Cor \ref{cor:positive} we get that $ff^\dagger$ is ($\dagger$-)Drazin if and only if $f^\dagger f$ is ($\dagger$-)Drazin. So to prove the desired statement, it remains to show that $(i) \Leftrightarrow (iii)$. 

For $(i) \Rightarrow (iii)$, suppose that $f$ is $\dagger$-Drazin. We will show that $(f f^\dagger)^D : = {f^\partial}^\dagger f^\partial$ satisfies the three axioms of a Drazin inverse. 
\begin{enumerate}[{\bf [D.1]}]
\item $(f f^\dagger)^{k+1} (f f^\dagger)^D  \stackrel{\text{def.}}{=}  (f f^\dagger)^{k+1}  {f^\partial}^\dagger f^\partial = (f f^\dagger)^{k} f f^\dagger {f^\partial}^\dagger f^\partial \stackrel{\text{Prop. \ref{prop:dag-Draz}.(\ref{prop:D2D3D4})}}{=} (f f^\dagger)^{k} f f^\partial \stackrel{\text{\textbf{[$\text{D}^{\dagger}$.6]}}}{=} (f f^\dagger)^{k} $
\item $(f f^\dagger)^D f f^\dagger (f f^\dagger)^D \stackrel{\text{def.}}{=} {f^\partial}^\dagger f^\partial f f^\dagger {f^\partial}^\dagger f^\partial \stackrel{\text{Prop. \ref{prop:dag-Draz}.(\ref{prop:D2D3D4})}}{=} {f^\partial}^\dagger f^\partial f f^\partial \stackrel{\text{\textbf{[$\text{D}^{\dagger}$.2]}}}{=} {f^\partial}^\dagger f^\partial \stackrel{\text{def.}}{=} (f f^\dagger)^D$ 
\item $ (f f^\dagger)^D f f^\dagger \stackrel{\text{def.}}{=} {f^\partial}^\dagger f^\partial f f^\dagger \stackrel{\text{Prop. \ref{prop:dag-Draz}.(\ref{prop:D2D3D4})}}{=} f f^\partial {f^\partial}^\dagger f^\dagger \stackrel{\text{Prop. \ref{prop:dag-Draz}.(\ref{prop:D2D3D4})}}{=} f f^\dagger {f^\partial}^\dagger f^\partial \stackrel{\text{def.}}{=}  f f^\dagger  (f f^\dagger)^D $
\end{enumerate}
So we conclude that $ff^\dagger$ Drazin. 

For $(iii) \Rightarrow (i)$, suppose that $f f^\dagger$ is Drazin. We will show that $f^\partial := f^\dagger (f f^\dagger)^D$ satisfies the necessary axioms to be a $\dagger$-Drazin inverse of $f$. 
\begin{enumerate}[{\bf [$\text{D}^{\dagger}$.1]}]
\setcounter{enumi}{5}
\item $(f f^\dagger)^k f f^\partial \stackrel{\text{def.}}{=}   (f f^\dagger)^k f f^\dagger (f f^\dagger)^D = (f f^\dagger)^{k+1} (f f^\dagger)^D \stackrel{\text{\textbf{[$\text{D}$.1]}}}{=} (f f^\dagger)^k$
\end{enumerate}
\begin{enumerate}[{\bf [$\text{D}^{\dagger}$.1]}]
\setcounter{enumi}{1}
\item $f^\partial f f^\partial \stackrel{\text{def.}}{=}   f^\dagger (f f^\dagger)^D f  f^\dagger (f f^\dagger)^D \stackrel{\text{\textbf{[$\text{D}$.2]}}}{=}  f^\dagger (f f^\dagger)^D  \stackrel{\text{def.}}{=} f^\partial$
\end{enumerate}
For the remaining two axioms, recall that by Lemma \ref{lem:Draz-dag}, since $ff^\dagger$ is self-adjoint then so is $(ff^\dagger)^D$. 
\begin{enumerate}[{\bf [$\text{D}^{\dagger}$.1]}]
\setcounter{enumi}{2}
\item $(ff^\partial)^\dagger \stackrel{\text{def.}}{=} (f f^\dagger (f f^\dagger)^D)^\dagger = f f^\dagger (f f^\dagger)^D \stackrel{\text{\textbf{[$\text{D}$.3]}}}{=} f f^\dagger (f f^\dagger)^D  \stackrel{\text{def.}}{=} f f^\partial$
\item $(f^\partial f)^\dagger \stackrel{\text{def.}}{=} (f^\dagger (f f^\dagger)^D f)^\dagger =  f^\dagger (f f^\dagger)^D f \stackrel{\text{def.}}{=}  f^\partial f$
\end{enumerate}
So we conclude that $f$ is $\dagger$-Drazin. Therefore, we have that $(i) \Leftrightarrow (ii) \Leftrightarrow (iii) \Leftrightarrow (iv)$ as desired. 

Now for (\ref{thm:vi}), we have already shown that $f^\partial = f^\dagger (f f^\dagger)^D$. For the other part of the equality we will need the only following fact: for maps $g: A \to B$ and $h: B \to A$ such that $gh$ and $hg$ are Drazin, then $(gh)^D g = g(hg)^D$ and $h(gh)^D = (hg)h^D$ \cite[Lemma 7.5]{cockett2024drazin}. Therefore, we indeed have that $f^\partial = f^\dagger (f f^\dagger)^D = (f^\dagger f)^D f^\dagger$. Then by Prop. \ref{prop:dag-Draz}.(\ref{prop:dag-Draz-adj}) and  Lemma \ref{lem:Draz-dag}, we also get that ${f^\dagger}^\partial = f (f^\dagger f)^D  =  (f f^\dagger)^D f$. For (\ref{thm:vii}), we have already shown that $(f f^\dagger)^D = {f^\partial}^\dagger f^\partial$, so by Prop. \ref{prop:dag-Draz}.(\ref{prop:dag-Draz-adj}) we get that $(f f^\dagger)^D = {f^\dagger}^\partial f^\partial$, and then applying the same formula to $f^\dagger$, we get that $(f^\dagger f)^D = f^\partial {f^\dagger}^\partial$. Lastly, by the arguments made in the above computations, it is straightforward to check that (\ref{thm:viii}) holds. 
\end{proof}

It immediately follows that, unsurprisingly, having a $\dagger$-group inverse corresponds to both induced positive maps having a group inverse. 

\begin{corollary} In a $\dagger$-category, a map $f$ has a $\dagger$-group inverse if and only if $f f^\dagger$ and $f^\dagger f$ have group inverses. 
\end{corollary}
\begin{proof} For the $\Rightarrow$ direction, suppose that $f$ has a $\dagger$-group inverse. Which implies that $f$ is $\dagger$-Drazin with $\mathsf{ind}^\partial(f) \leq 1$. Then by Thm \ref{thm:Drazin=dag-Drazin}, we have that $ff^\dagger$ and $f^\dagger f$ are Drazin, and also that $\mathsf{ind}^D(ff^\dagger) \leq \mathsf{ind}^\partial(f) \leq 1$ and $\mathsf{ind}^D(f^\dagger f) \leq \mathsf{ind}^\partial(f) \leq 1$. This implies that $f f^\dagger$ and $f^\dagger f$ have group inverses. Conversely, for the $\Leftarrow$ direction, suppose that $f f^\dagger$ and $f^\dagger f$ both have group inverses. So they are both Drazin and $\mathsf{ind}^D(ff^\dagger) \leq 1$ and $\mathsf{ind}^D(f^\dagger f)  \leq 1$. Then by Thm \ref{thm:Drazin=dag-Drazin}, we have that $f$ is $\dagger$-Drazin and $\mathsf{ind}^\partial(f) = \mathsf{max}\lbrace \mathsf{ind}^D(ff^\dagger), \mathsf{ind}^D(f^\dagger f) \rbrace \leq 1$. Therefore, $f$ has a $\dagger$-group inverse. 
\end{proof}

We also obtain an equivalent characterization of $\dagger$-Drazin categories: 

\begin{corollary} A $\dagger$-category is $\dagger$-Drazin if and only if every positive map is Drazin. 
\end{corollary}

This also allows us to conclude that a Drazin category, which has a  $\dagger$, is automatically a $\dagger$-Drazin category. 

\begin{corollary}\label{cor:Drazin-Dagger-implies-Dagger-Drazin} Every Drazin $\dagger$-category (by which we mean a $\dagger$-category whose underlying category is Drazin) is $\dagger$-Drazin. In particular, if $\mathbb{X}$ is a Drazin category, then for any dagger $\dagger$ on $\mathbb{X}$, $\mathbb{X}$ is $\dagger$-Drazin. 
\end{corollary}

We are now in a position to give examples of $\dagger$-Drazin inverses. Our first examples are matrices over a field. For any field, every matrix over the field has a $\dagger$-Drazin inverse with respect to the transpose operator. Moreover, a matrix over an involutive field also has $\dagger$-Drazin inverse with respect to the conjugate transpose operator. 

\begin{example}\label{ex:matrices-field} For a field $\mathbb{K}$, let $\mathsf{MAT}(\mathbb{K})$ be the category whose objects are natural numbers $n \in \mathbb{N}$ and where a map $A: n \to m$ is an $n \times m$ $\mathbb{K}$-matrix. Every square $\mathbb{K}$-matrix has a Drazin inverse, and therefore $\mathsf{MAT}(\mathbb{K})$ is a Drazin category -- see \cite[Sec 3.2]{cockett2024drazin} for how to construct the Drazin inverse. Moreover, the Drazin index of a square matrix $B$ corresponds precisely to the \textbf{index} of $B$ \cite[Def 7.2.1]{campbell2009generalized}, that is, the least $k \in \mathbb{N}$ such that $\mathsf{rank}(B^{k+1}) = \mathsf{rank}(B^k)$. On the other hand, $\mathsf{MAT}(\mathbb{K})$ is a $\dagger$-category whose dagger is the transpose operator $\mathsf{T}$, $A^\mathsf{T}(i,j) = A(j,i)$. As such, $\mathsf{MAT}(\mathbb{K})$ is a Drazin $\dagger$-category, and so in particular also $\dagger$-Drazin. So every $\mathbb{K}$-matrix $A$ has a $\mathsf{T}$-Drazin inverse given by $A^\partial = A^\mathsf{T} (A A^\mathsf{T})^D$. Furthermore, the $\mathsf{T}$-Drazin index of a matrix $A$ is the maximum between the index of $AA^\mathsf{T}$ and the index of $A^\mathsf{T}A$. 
\end{example}

\begin{example}\label{ex:matrices-inv-field} If $\mathbb{K}$ is an involutive field, with involution $x \mapsto \overline{x}$, then $\mathsf{MAT}(\mathbb{K})$ has another dagger this time given by the conjugate transpose operator $\ast$, $A^\ast(i,j) = \overline{A(j,i)}$. Then $\mathsf{MAT}(\mathbb{K})$ is still a Drazin $\dagger$-category with respect to $\ast$, and so in particular also $\dagger$-Drazin. So every $\mathbb{K}$-matrix $A$ has a $\ast$-Drazin inverse given by $A^\partial = A^\ast (A A^\ast)^D$, and where the $\ast$-Drazin index of $A$ is the maximum between the index of $AA^\ast$ and the index of $A^\ast A$. We will revisit this example when $\mathbb{K}$ is the involutive field of complex numbers in Ex \ref{ex:complex-matrices} below. 
\end{example}

The converse of Cor \ref{cor:Drazin-Dagger-implies-Dagger-Drazin} is not true, that is, there are $\dagger$-Drazin categories that are not Drazin. An example arises from considering \emph{inverse categories}. 

\begin{example}\label{ex:inv-cat} An \textbf{inverse category} may be viewed as a $\dagger$-category in which every map is a partial isometry (that is, $f f^\dagger f = f$) and for all parallel maps $f$ and $g$, $ff^\dagger gg^\dagger = g g^\dagger ff^\dagger$. Recall that by Lemma \ref{lem:dag-Draz-maps}.(\ref{lem:dag-Draz-partial-isom}), every partial isometry is $\dagger$-Drazin with $\dagger$-Drazin inverse being its adjoint and $\dagger$-Drazin index less than or equal to $1$. Therefore every inverse category is $\dagger$-Drazin. 
\end{example}

\begin{example}\label{ex:PINJ}
Let $\mathsf{PINJ}$ be the category of sets and partial injections. Then $\mathsf{PINJ}$ is an inverse category where for a partial injection $f: X \to Y$, $f^\circ: Y \to X$ is defined as $f^\circ(y) = x$ if $f(x)=y$ and is undefined otherwise. As $\mathsf{PINJ}$ is an inverse category it is certainly $\dagger$-Drazin, and so every partial injection $f$ is $\dagger$-Drazin with $f^\partial = f^\circ$. However, $\mathsf{PINJ}$ is not Drazin. Consider the successor function $s: \mathbb{N} \to \mathbb{N}$, $s(n) = n+1$. Suppose that $s$ had a Drazin inverse $s^D: \mathbb{N} \to \mathbb{N}$ and $\mathsf{ind}(s) =k$. By \textbf{[D.3]}, we would have that $s^D(n) = s^D(s^n(0)) =  s^n(s^D(0)) = s^D(0) + n$. So $s^D$ is completely determined by $s^D(0)$. If $s^D(0)$ is defined, then by \textbf{[D.1]} we compute that $k = s^k(0) = s^D(s^{k+1}(0)) = s^D(k+1) = s^D(0) + k +1$. This implies that $0 = s^D(0) + 1$ -- which is a contradiction since $s^D(0) \in \mathbb{N}$. On the other hand if $s^D(0)$ is undefined, then $s^k(0) = s^D(s^{k+1}(0)) = s^{k+1}(s^D(0)) $ would also be undefined, but this is a contradiction since $s^k$ is total. So the successor function does not have a Drazin inverse, and therefore $\mathsf{PINJ}$ is not Drazin. 
\end{example}

Even in $\dagger$-categories that are neither Drazin nor $\dagger$-Drazin, it is sometimes possible to characterize the maps that are $\dagger$-Drazin. 

\begin{example}\label{ex:Hilb} Let $\mathsf{HILB}$ be the category of (complex) Hilbert spaces and bounded linear operators between them. Then $\mathsf{HILB}$ is a $\dagger$-category where the dagger is given by the adjoint $\ast$, that is, for a bounded linear operator $f: H_1 \to H_2$, $f^\ast: H_2 \to H_1$ is the unique bounded linear operator such that $\langle f(x) \vert y \rangle  =\langle x \vert f^\ast(y) \rangle$ for all $x \in H_1$ and $y \in H_2$. $\mathsf{HILB}$ is not Drazin, nor is it $\dagger$-Drazin. Nevertheless, we can still characterize the $\ast$-Drazin maps in $\mathsf{HILB}$, since there is a characterization of when a bounded linear operator is Drazin. For a bounded linear operator $f: H_1 \to H_2$, let $\mathsf{N}(f)$ denote its null space and $\mathsf{R}(f)$ its range space. A bounded linear operator $g: H \to H$ is said to be has \textbf{finite ascent} (resp. \textbf{descent}) if there exists a $k \in \mathbb{N}$ such that $\mathsf{N}(g^k) = \mathsf{N}(g^{k+1})$ (resp. $\mathsf{R}(g^k) = \mathsf{R}(g^{k+1})$), and the least such $k$ is called the \textbf{ascent} (resp. \textbf{descent}) of $g$. Then a bounded linear operator $g: H \to H$ is Drazin if and only if it has both finite ascent and finite descent \cite{King1977ANO,wei2003representation}, and furthermore its Drazin index corresponds to its ascent (or equivalently its descent, which are equal in this). Therefore, it follows that a bounded linear operator $f: H_1 \to H_2$ is $\ast$-Drazin if and only if $ff^\ast$ (or equivalently $f^\ast f$) has both finite ascent and finite descent. In this case, the $\ast$-Drazin of $f$ is equal to the maximum of the ascent (or equivalently descent) of $ff^\ast$ or $f^\ast f$. 
\end{example}

\section{Comparing Dagger-Drazin and Moore-Penrose}

In this section, we explain the relationship between $\dagger$-Drazin inverses and Moore-Penrose inverses. In particular, the main result of this section is that having a Moore-Penrose inverse is equivalent to being \emph{a} $\dagger$-Drazin inverse (which we stress is different than simply being $\dagger$-Drazin inverse and having a $\dagger$-Drazin inverse). For a detailed introduction to Moore-Penrose inverses in dagger categories, we invite the reader to see \cite{EPTCS384.10}.

\begin{definition} \cite[Def 2.3]{EPTCS384.10} In a $\dagger$-category,  a \textbf{Moore-Penrose inverse} of a map $f: A \to B$ is a map of dual type $f^\circ: B \to A$ such that:
%\vspace{-5pt}
\begin{enumerate}[{\bf [$\text{MP}$.1]}]
%\begin{multicols}{4}
    \item $f f^\circ f = f$
%\columnbreak
    \item $f^\circ f f^\circ = f^\circ$
%\columnbreak
\item $(ff^\circ)^\dagger = ff^\circ$
%\columnbreak
    \item $(f^\circ f)^\dagger = f^\circ f$
%\end{multicols}
\end{enumerate}
%\vspace{-10pt}
If $f$ has a Moore-Penrose inverse, we say that $f$ is \textbf{Moore-Penrose invertible} or simply \textbf{Moore-Penrose}. A \textbf{Moore-Penrose $\dagger$-category} is a $\dagger$-category such that every map is Moore-Penrose. 
\end{definition}

Moore-Penrose inverses are \emph{unique} \cite[Lemma 2.4]{EPTCS384.10}, so again we may speak of \emph{the} Moore-Penrose inverse of a map and that being Moore-Penrose is a property rather than structure. Many examples and properties of Moore-Penrose inverses and Moore-Penrose dagger categories can be found in \cite{EPTCS384.10}. 

We now show that having a Moore-Penrose inverse is equivalent to being a $\dagger$-Drazin inverse. It is important to note that being a $\dagger$-Drazin inverse is strictly stronger than being $\dagger$-Drazin. Indeed, while every $\dagger$-Drazin inverse is $\dagger$-Drazin, not every $\dagger$-Drazin map is itself a $\dagger$-Drazin inverse. A $\dagger$-Drazin map is a $\dagger$-Drazin inverse precisely when it is the $\dagger$-Drazin inverse of its $\dagger$-Drazin inverse.  

\begin{theorem}\label{thm:MP=a-dag-Drazin} In a $\dagger$-category, the following are equivalent for a map $f$:
   \begin{enumerate}[(i)]
\item $f$ is Moore-Penrose;
\item $f$ is a $\dagger$-Drazin inverse;
\item $f$ is $\dagger$-Drazin and $f^{\partial \partial} = f$. 
\end{enumerate}
Therefore, if $f$ is Moore-Penrose then $f$ is $\dagger$-Drazin where $f^\partial = f^\circ$ and $\mathsf{ind}^\partial(f)\leq 1$ (in other words, $f^\circ$ is the $\dagger$-group inverse of $f$), and, moreover, $f$ is the $\dagger$-Drazin inverse of $f^\circ$, that is, ${f^\circ}^\partial = f$. 
\end{theorem}
\begin{proof} We will prove $(i) \Rightarrow (iii) \Rightarrow (ii) \Rightarrow (i)$. So for $(i) \Rightarrow (iii)$, suppose that $f$ is Moore-Penrose. We will show that its Moore-Penrose inverse $f^\circ$ is also its $\dagger$-Drazin inverse. However observe that \textbf{[$\text{D}^{\dagger}$.2]} is \textbf{[$\text{MP}$.1]}, \textbf{[$\text{D}^{\dagger}$.3]} is \textbf{[$\text{MP}$.4]}, and \textbf{[$\text{D}^{\dagger}$.4]} is \textbf{[$\text{MP}$.3]}. It remains to show \textbf{[$\text{D}^{\dagger}$.5]}. So set $j=1$, then we easily have that $f f^\partial f f^\dagger \stackrel{\text{def.}}{=} f f^\circ f f^\dagger \stackrel{\text{\textbf{[$\text{MP}$.1]}}}{=}  f f^\dagger$. So $f^\circ$ is the $\dagger$-Drazin inverse of $f$. Lastly, by Prop \ref{prop:dag-Draz}.(\ref{prop:dag-Draz-dag-Draz}), $f^{\partial \partial} = f$ is precisely \textbf{[$\text{MP}$.1]}. Next note $(iii) \Rightarrow (ii)$ is automatic by assumption that $f^{\partial \partial} = f$. So it remains to show that $(ii) \Rightarrow (i)$. So suppose that $f$ is a $\dagger$-Drazin inverse for map $g$, that is, $f = g^\partial$. By Prop \ref{prop:dag-Draz}.(\ref{prop:dag-Draz-dag-Draz}), we know that $f$ is also $\dagger$-Drazin where $f^\partial = g f g$. We will show that $f^\partial$ is also the Moore-Penrose inverse of $f$. To do so, recall that by Prop \ref{prop:dag-Draz}.(\ref{prop:dag-Draz-dag-Draz-dag-Drax}) that $f^{\partial \partial} = g^{\partial \partial \partial} = g^\partial =f$. So $f$ is the $\dagger$-Drazin inverse of $f^\partial$. Since $f$ is the $\dagger$-Drazin inverse of $f^\partial$, by \textbf{[$\text{D}^{\dagger}$.2]} we have that $f f^\partial f = f$, which is \textbf{[MP.1]}. On the other hand, since $f^\partial$ is the $\dagger$-Drazin inverse of $f$, \textbf{[$\text{D}^{\dagger}$.2]} gives us $f^\partial f f^\partial= f^\partial$, which is \textbf{[MP.2]}. Lastly, \textbf{[MP.3]} and \textbf{[MP.4]} are precisely \textbf{[$\text{D}^{\dagger}$.3]} and \textbf{[$\text{D}^{\dagger}$.4]} as well. So $f$ is Moore-Penrose. 
\end{proof}

This also allows us to give an equivalent characterization of Moore-Penrose $\dagger$ categories as special kind of $\dagger$-Drazin category. 

\begin{corollary} A $\dagger$-category is Moore-Penrose if and only if it is $\dagger$-Drazin where every map is also a $\dagger$-Drazin inverse.
\end{corollary}

We already know from Ex \ref{ex:matrices-inv-field} that complex matrices have $\dagger$-Drazin inverses (with respect to the conjugate transpose dagger). However we now have an alternative explanation as to why, and an alternative description of its $\dagger$-Drazin inverse from the perspective of Moore Penrose inverse. 

\begin{example}\label{ex:complex-matrices} Let $\mathbb{C}$ be the field of complex numbers. $\mathsf{MAT}(\mathbb{C})$ with the conjugate transpose dagger $\ast$ is a Moore-Penrose $\dagger$-category where the Moore-Penrose inverse of a complex matrix can be constructed from its singular value decomposition \cite[Ex 2.9]{EPTCS384.10}. Therefore, every complex matrix is $\ast$-Drazin and its $\ast$-Drazin inverse is its Moore-Penrose inverse. Of course, we can also say that every complex matrix is the $\ast$-Drazin inverse of its Moore-Penrose inverse. 
\end{example}

It is important to note that while for any (involutive) field, its category of matrices is always $\dagger$-Drazin, it may not be Moore-Penrose. 

\begin{example} Consider $\mathsf{MAT}(\mathbb{C})$ this time with simply the transpose $\mathsf{T}$ dagger, which we know is $\dagger$-Drazin from Ex \ref{ex:matrices-field}, that is, every complex matrix has a $\mathsf{T}$-Drazin inverse. However, $\mathsf{MAT}(\mathbb{C})$ is not Moore-Penrose with the transpose dagger, in particular, the matrix $\begin{bmatrix} i & 1 \end{bmatrix}$ does not have a Moore-Penrose inverse with respect to the transpose \cite[Ex 2.10]{EPTCS384.10}. However $\begin{bmatrix} i & 1 \end{bmatrix}$ does have a $\mathsf{T}$-Drazin inverse. In fact, since $\begin{bmatrix} i & 1 \end{bmatrix}\begin{bmatrix} i & 1 \end{bmatrix}^\mathsf{T}=0$, it follows that $\begin{bmatrix} i & 1 \end{bmatrix}^\partial = \begin{bmatrix} 0 & 0 \end{bmatrix}^\mathsf{T}$. 
\end{example}

Recall that in Ex \ref{ex:Hilb}, we characterized $\dagger$-Drazin maps in the category of Hilbert spaces. We can now take a look at the subcategory of finite dimensional Hilbert spaces. 

\begin{example}\label{ex:FHILB} $\mathsf{HILB}$ is not Moore-Penrose but we can again characterize when bounded linear maps are Moore-Penrose -- they are precisely those whose range is closed \cite[Ex 2.12]{EPTCS384.10}. Moreover, for appropriate bounded linear maps, it is sometimes possible to formulate the Drazin inverse (and hence the $\dagger$-Drazin inverse) in terms of the Moore-Penrose inverse \cite[Lemma 2.1]{wei2003representation}. Now let $\mathsf{FHILB}$ be the subcategory of finite dimensional Hilbert spaces. $\mathsf{FHILB}$ is a Moore-Penrose $\dagger$-category \cite[Ex 2.12]{EPTCS384.10}, and therefore $\mathsf{FHILB}$ is also $\dagger$-Drazin. So every linear operator between finite dimensional Hilbert spaces is $\ast$-Drazin, where its $\ast$-Drazin inverse is its Moore-Penrose inverse. 
\end{example}

\section{Revisiting Drazin Opposing Pairs}

In \cite{cockett2024drazin}, the authors introduced the notion of Drazin inverses for opposing pairs, which was the authors first attempt to generalize the notion of Drazin inverses to maps of arbitrary type. In section, we relate Drazin inverse of opposing pairs to $\dagger$-Drazin inverses. In particular, we show that Drazin inverse of opposing pairs correspond precisely to $\dagger$-Drazin inverses in \emph{cofree} dagger categories. %This is very cool :)

It is useful to understand the original motivation for considering Drazin inverses of opposing pairs. A fundamental idea behind Drazin inverses is that certain endomorphisms can be undone after iterating them a certain number of times. Thus if one wishes to give a generalization of a Drazin inverse for a map $f: A \to B$ in an arbitrary category, then one also needs a map of dual type $g: B \to A$ so that the Drazin iteration axiom can be expressed. Thus, in an arbitrary category $\mathbb{X}$, an \textbf{opposing pair} $(f,g): A \to B$ is a pair $(f,g)$ consisting of maps of dual type, so $f: A \to B$ and $g: B \to A$. Such an opposing pair, $(f,g): A \to B$, induces two endomorphism $fg: A \to A$ and $gf: B \to B$ for which we can seek Drazin inverses. These in turn -- if they exist -- induce an opposing pair in the opposite direction.  This provides a notion of inverse for an arbitrary pairs of opposing maps. 

\begin{definition}\cite[Def 7.1]{cockett2024drazin} \label{def:opposing-pair-drazin-inverse} In a category $\mathbb{X}$, a \textbf{Drazin inverse} of an opposing pair $(f,g): A \to B$ is an opposing pair of dual type $(f,g)^D = (f^{\frac{D}{g}},g^{\frac{D}{f}}): B \to A$ such that:
\begin{enumerate}[{\bf [DV.1]}]
\item There is a $k \in \mathbb{N}$ such that $(fg)^k f f^{\frac{D}{g}} = (fg)^k$ and $(gf)^k g g^{\frac{D}{f}} = (gf)^k$;
%\begin{multicols}{2}
\item $f^{\frac{D}{g}} f f^{\frac{D}{g}} = f^{\frac{D}{g}}$ and $g^{\frac{D}{f}} g g^{\frac{D}{f}} = g^{\frac{D}{f}}$;
%\columnbreak
\item $ f f^{\frac{D}{g}} =  g^{\frac{D}{f}} g$ and $ f^{\frac{D}{g}} f =  g g^{\frac{D}{f}}$. 
%\end{multicols}
\end{enumerate}
%\vspace{-10pt}
 We say that an opposing pair $(f,g)$ is \textbf{Drazin} if it has a Drazin inverse. The map $f^{\frac{D}{g}}: B \to A$ is called the \textbf{Drazin inverse of $f$ over $g$}, while the map $g^{\frac{D}{f}}: A \to B$ is called the \textbf{Drazin inverse of $g$ over $f$}. The \textbf{Drazin index} of $(f,g)$ is the smallest $k \in \mathbb{N}$ such that {\em \textbf{[DV.1]}} holds, which we denote by $\mathsf{ind}^D(f,g)=k$.
\end{definition}

Drazin inverses of opposing pairs share many of the same properties as Drazin inverses of endomorphisms, see \cite[Sec 7]{cockett2024drazin}. In particular, the Drazin inverse of an opposing pair is unique \cite[Prop 7.7]{cockett2024drazin}. One can also recover Drazin inverses of endomorphisms by considering the Drazin inverse of the opposing pair consisting of an endomorphism and the identity \cite[Lemma 7.10]{cockett2024drazin}. Moreover, since every opposing pair induces two endomorphisms, it turns out that Drazin opposing pairs can also be characterized in terms of Drazin inverses.

\begin{theorem}\label{thm:Draz-opp}\cite[Thm 7.6]{cockett2024drazin} In a category $\mathbb{X}$, the following are equivalent for an opposing pair $(f,g)$: 
    \begin{enumerate}[(i)]
    %\begin{multicols}{3}
\item $(f,g)$ is Drazin; 
%\columnbreak
\item $fg$ is Drazin;
%\columnbreak
\item $gf$ is Drazin. 
%\end{multicols}
\end{enumerate}
In particular, if $(f,g)$ is Drazin then \cite[Cor 7.8]{cockett2024drazin}:
    \begin{enumerate}[(i)]
    \setcounter{enumi}{3}
\item \label{thm:4} $(fg)^D = g^{\frac{D}{f}} f^{\frac{D}{g}}$ and $(gf)^D = f^{\frac{D}{g}}g^{\frac{D}{f}}$; 
\item \label{thm:5} $f^{\frac{D}{g}} = g (fg)^D = (gf)^D g$ and $g^{\frac{D}{f}} = f (gf)^D = (fg)^D f$;
\item \label{thm:6} $\mathsf{ind}^D(f,g) = \mathsf{max}\lbrace \mathsf{ind}^D(fg),\mathsf{ind}^D(gf) \rbrace$. 
\end{enumerate}
\end{theorem}

Now in a $\dagger$-category, every map and its adjoint form an opposing pair, which we call an \textbf{adjoint opposing pair}. Then, it is straightforward to see that a map being $\dagger$-Drazin is equivalent to its adjoint opposing pair being Drazin. 

\begin{proposition} In a $\dagger$-category, a map $f$ is $\dagger$-Drazin if and only if the opposing pair $(f,f^\dagger)$ is Drazin. Moreover, if $f$ is $\dagger$-Drazin then $f^{\frac{D}{f^\dagger}} = f^\partial$ and ${f^\dagger}^{\frac{D}{f}} = {f^\partial}^\dagger$, and also that $\mathsf{ind}^\partial(f) = \mathsf{ind}^D(f,g)$. 
\end{proposition}
\begin{proof} By combining Thm \ref{thm:Drazin=dag-Drazin} and Thm \ref{thm:Draz-opp}, we get that $f$ is $\dagger$-Drazin if and only if $ff^\dagger$ is Drazin if and only if $(f,f^\dagger)$ is Drazin. By Thm \ref{thm:Drazin=dag-Drazin}.(\ref{thm:vi}) and Thm \ref{thm:Draz-opp}.(\ref{thm:5}), we get that $f^{\frac{D}{f^\dagger}} = f^\dagger (f f^\dagger)^D = f^\partial$ and ${f^\dagger}^{\frac{D}{f}} = f (f^\dagger f)^D = {f^\dagger}^\partial = {f^\partial}^\dagger$. Lastly, Thm \ref{thm:Drazin=dag-Drazin}.(\ref{thm:viii}) and Thm \ref{thm:Draz-opp}.(\ref{thm:6}) give us that $\mathsf{ind}^\partial(f) =  \mathsf{max}\lbrace \mathsf{ind}^D(ff^\dagger),\mathsf{ind}^D(f^\dagger f) \rbrace = \mathsf{ind}^D(f,g)$. 
\end{proof}

We now turn our attention to proving that Drazin inverses of opposing pairs correspond to the $\dagger$-Drazin inverses in \emph{cofree} dagger categories. This makes sense since maps in cofree dagger categories are precisely opposing pairs \cite[Def 3.1.16]{heunen2009categorical}. Indeed, given a category $\mathbb{X}$, let $\mathbb{X}^\leftrightarrows$ be the category that has the same objects as $\mathbb{X}$ and where a map is an opposing pair $(f,g): A \to B$ in $\mathbb{X}$. Composition is defined as $(f,g)(h,k) = (fh, kg)$ and the identity is $(1_A, 1_A)$. Then $\mathbb{X}^\leftrightarrows$ is a $\dagger$-category where $(f,g)^\dagger = (g,f)$. Moreover $\mathbb{X}^\leftrightarrows$ is the cofree $\dagger$-category over $\mathbb{X}$ \cite[Thm 3.1.17]{heunen2009categorical}. 

\begin{theorem}\label{thm:dag-Drazin-cofree} In a category $\mathbb{X}$, an opposing pair $(f,g)$ in $\mathbb{X}$ is Drazin if and only if $(f,g)$ is $\dagger$-Drazin in $\mathbb{X}^\leftrightarrows$. Moreover in this case, $(f,g)^D = (f,g)^\partial$ and $\mathsf{ind}^D(f,g) = \mathsf{ind}^\partial(f,g)$.  
\end{theorem}
\begin{proof} For the $\Rightarrow$ direction, suppose that $(f,g)$ is Drazin in $\mathbb{X}$. We show that $(f,g)^\partial = (f^{\frac{D}{g}},g^{\frac{D}{f}})$ is the $\dagger$-Drazin inverse of $(f,g)$ in $\mathbb{X}^\leftrightarrows$. We begin by checking {\bf [$\text{D}^{\dagger}$.2]}--{\bf [$\text{D}^{\dagger}$.4]}. 
\begin{enumerate}[{\bf [$\text{D}^{\dagger}$.1]}]
\setcounter{enumi}{1}
\item $(f,g)^\partial (f,g) (f,g)^\partial \stackrel{\text{def.}}{=} (f^{\frac{D}{g}},g^{\frac{D}{f}}) (f,g) (f^{\frac{D}{g}},g^{\frac{D}{f}}) = (f^{\frac{D}{g}} f f^{\frac{D}{g}}, g^{\frac{D}{f}} g g^{\frac{D}{f}}) \stackrel{\text{\textbf{[$\text{DV}$.2]}}}{=} (f^{\frac{D}{g}},g^{\frac{D}{f}}) \stackrel{\text{def.}}{=} (f,g)^\partial$
\item $(f,g) (f,g)^\partial \stackrel{\text{def.}}{=} (f,g) (f^{\frac{D}{g}},g^{\frac{D}{f}}) =  (f f^{\frac{D}{g}}, g^{\frac{D}{f}}g) \stackrel{\text{\textbf{[$\text{DV}$.3]}}}{=} ( g^{\frac{D}{f}}g, f f^{\frac{D}{g}}) = (f f^{\frac{D}{g}}, g^{\frac{D}{f}}g)^\dagger =  \left( (f,g) (f,g)^\partial \right)^\dagger$
\item $(f,g)^\partial (f,g) \stackrel{\text{def.}}{=} (f^{\frac{D}{g}},g^{\frac{D}{f}}) (f,g) = ( f^{\frac{D}{g}} f, g g^{\frac{D}{f}}) \stackrel{\text{\textbf{[$\text{DV}$.3]}}}{=} ( g g^{\frac{D}{f}},  f^{\frac{D}{g}} f) = (f^{\frac{D}{g}}f,gg^{\frac{D}{f}})^\dagger = \left(  (f,g)^\partial (f,g) \right)^\dagger$
\end{enumerate}
For {\bf [$\text{D}^{\dagger}$.6]}, we first note that $fg f f^{\frac{D}{g}}\stackrel{\text{\textbf{[$\text{DV}$.3]}}}{=} f g  g^{\frac{D}{f}}g \stackrel{\text{\textbf{[$\text{DV}$.3]}}}{=}  f f^{\frac{D}{g}} f g$. So $fg f f^{\frac{D}{g}} = f f^{\frac{D}{g}} f g$. Therefore \textbf{[DV.1]} tells us that there is a $k \in \mathbb{N}$ such that $(fg)^k f f^{\frac{D}{g}} = (fg)^k = f f^{\frac{D}{g}}(fg)^k$. So then we compute that:
\begin{gather*}
    \left( (f,g)(f,g)^\dagger \right)^k (f,g) (f,g)^\partial \stackrel{\text{def.}}{=} \left( (f,g) (g,f) \right)^k (f,g) (f^{\frac{D}{g}},g^{\frac{D}{f}}) = \left( fg, fg \right)^k (f f^{\frac{D}{g}},g^{\frac{D}{f}} g) \\ 
    = ( (fg)^k, (fg)^k) (f f^{\frac{D}{g}},g^{\frac{D}{f}} g) = ( (fg)^k f f^{\frac{D}{g}}, g^{\frac{D}{f}}g (fg)^k)\\
\stackrel{\text{\textbf{[$\text{DV}$.3]}}}{=}  ( (fg)^k f f^{\frac{D}{g}}, f f^{\frac{D}{g}} (fg)^k)  \stackrel{\text{\textbf{[$\text{DV}$.1]}}}{=} ( (fg)^k, (fg)^k) =  \left( (f,g)(f,g)^\dagger \right)^k
\end{gather*}
So {\bf [$\text{D}^{\dagger}$.6]} holds. So we conclude that $(f,g)$ is $\dagger$-Drazin in $\mathbb{X}^\leftrightarrows$.

For the $\Leftarrow$ direction, suppose that $(f,g)$ is $\dagger$-Drazin in $\mathbb{X}^\leftrightarrows$. Let $(f,g)^\partial$ be its $\dagger$-Drazin, which we denote as $(f,g)^\partial = (f^{\frac{D}{g}},g^{\frac{D}{f}})$. We will show that $(f,g)$ is Drazin in $\mathbb{X}$ where $(f^{\frac{D}{g}},g^{\frac{D}{f}})$ is its Drazin inverse. To do so, let us expand out the $\dagger$-Drazin inverse axioms. Starting with {\bf [$\text{D}^{\dagger}$.2]}, we get:
\[ (f^{\frac{D}{g}},g^{\frac{D}{f}}) = (f,g)^\partial \stackrel{\text{\textbf{[$\text{D}^{\dagger}$.2]}}}{=} (f,g)^\partial (f,g) (f,g)^\partial = (f^{\frac{D}{g}},g^{\frac{D}{f}})(f,g)(f^{\frac{D}{g}},g^{\frac{D}{f}}) = (f^{\frac{D}{g}}ff^{\frac{D}{g}}, g^{\frac{D}{f}}gg^{\frac{D}{f}}) \]
This give us that $f^{\frac{D}{g}} = f^{\frac{D}{g}}ff^{\frac{D}{g}}$ and $g^{\frac{D}{f}}=g^{\frac{D}{f}}gg^{\frac{D}{f}}$, which is \textbf{[$\text{DV}$.2]}. Next we expand out {\bf [$\text{D}^{\dagger}$.3]} and {\bf [$\text{D}^{\dagger}$.4]}:
\[ (f f^{\frac{D}{g}},  g^{\frac{D}{f}} g) = (f,g)  (f^{\frac{D}{g}},  g^{\frac{D}{f}}) = (f,g) (f,g)^\partial \stackrel{\text{\textbf{[$\text{D}^{\dagger}$.3]}}}{=} \left( (f,g) (f,g)^\partial \right)^\dagger = (f f^{\frac{D}{g}},  g^{\frac{D}{f}} g)^\dagger = (g^{\frac{D}{f}} g, f f^{\frac{D}{g}})  \]
\[ (f^{\frac{D}{g}} f,  g g^{\frac{D}{f}}) =  (f^{\frac{D}{g}},  g^{\frac{D}{f}}) (f,g) =  (f,g)^\partial (f,g) \stackrel{\text{\textbf{[$\text{D}^{\dagger}$.4]}}}{=} \left(  (f,g)^\partial (f,g) \right)^\dagger = (f^{\frac{D}{g}} f,  g g^{\frac{D}{f}})^\dagger = (g g^{\frac{D}{f}}, f^{\frac{D}{g}} f)  \]
    So we get that $f f^{\frac{D}{g}} =  g^{\frac{D}{f}} g$ and $f^{\frac{D}{g}} f =  g g^{\frac{D}{f}}$, so \textbf{[$\text{DV}$.3]} holds. Next we expand out {\bf [$\text{D}^{\dagger}$.6]} and {\bf [$\text{D}^{\dagger}$.7]}: 
\begin{gather*}
    ((fg)^k, (fg)^k) = \left( (f,g)(f,g)^\dagger \right)^k \stackrel{\text{\textbf{[$\text{D}^{\dagger}$.6]}}}{=} \left( (f,g)(f,g)^\dagger \right)^k (f,g)(f,g)^\partial = \left( (f,g) (g,f) \right)^k (f,g) (f^{\frac{D}{g}}, g^{\frac{D}{f}})\\
    = (fg, fg)^k (f  f^{\frac{D}{g}}, g^{\frac{D}{f}} g) = ((fg)^k, (fg)^k) (f  f^{\frac{D}{g}}, g^{\frac{D}{f}} g) = ((fg)^k f  f^{\frac{D}{g}}, g^{\frac{D}{f}} g (fg)^k) = ((fg)^k f  f^{\frac{D}{g}},  f  f^{\frac{D}{g}} (fg)^k)
\end{gather*}
\begin{gather*} ((gf)^k, (gf)^k) = \left( (f,g)^\dagger (f,g) \right)^k = (f,g)^\partial (f,g) \left( (f,g)^\dagger (f,g) \right)^k  \stackrel{\text{\textbf{[$\text{D}^{\dagger}$.7]}}}{=} (f^{\frac{D}{g}}, g^{\frac{D}{f}}) (f,g) \left( (g,f) (f,g) \right)^k\\
= ( f^{\frac{D}{g}} f, g g^{\frac{D}{f}}) \left( gf, gf \right)^k = ( f^{\frac{D}{g}} f, g g^{\frac{D}{f}}) ((gf)^k, (gf)^k) = ( f^{\frac{D}{g}} f (gf)^k, (gf)^k g g^{\frac{D}{f}}) = (g g^{\frac{D}{f}} (gf)^k, (gf)^k g g^{\frac{D}{f}})  \end{gather*}
So we get that $(fg)^k f f^{\frac{D}{g}} = (fg)^k$ and $(gf)^k g g^{\frac{D}{f}} = (gf)^k$, so \textbf{[$\text{DV}$.1]} holds. So we conclude that $(f,g)$ is Drazin in $\mathbb{X}$.  
\end{proof}

%\section{Conclusion}

\newpage 
\nocite{*}
\bibliographystyle{eptcs}
\bibliography{generic}

\end{document}